\documentclass[reqno]{amsart}
\usepackage{amssymb}
\usepackage{amsmath}
\usepackage{amsthm}
\usepackage{latexsym}
\usepackage{graphicx}
 \usepackage{cite}
 \numberwithin{equation}{section}

\newcommand{\w}{\mathbf{w}}
\newcommand{\bb}{\mathbf{b}}
\newcommand{\kk}{\mathbf{\kappa}}
\newtheorem{Theorem}{Theorem}[section]
\newtheorem{Lemma}[Theorem]{Lemma}

\newtheorem{Remark}{Remark}[section]
\allowdisplaybreaks

\begin{document}

\title[compressible magnetohydrodynamic equations with vacuum]
{Global strong solutions to the planar compressible magnetohydrodynamic equations with large initial data and vacuum}

\date {\today}
\author{Jishan Fan}
\address{ Department of Applied Mathematics,
 Nanjing Forestry University, Nanjing  210037, P.R.China}
\email{fanjishan@njfu.edu.cn}

\author{Shuxiang Huang}
 \address{School of Mathematics, Shandong University, Jinan 250100, P.R. China}
 \email{huangs@sdu.edu.cn}

\author[Fucai Li]{ Fucai Li}
\address{ Department of Mathematics,
  Nanjing University, Nanjing 210093, P.R.China}
\email{fli@nju.edu.cn}
%%%%----------------------------------------------------------------------------------%%%%%%

\begin{abstract}
This paper considers the initial boundary problem to
 the planar compressible magnetohydrodynamic equations with large initial data and vacuum. The global existence and uniqueness of
  large strong solutions are established when the heat conductivity  coefficient $\kappa(\theta)$
satisfies
\begin{equation*}
C_{1}(1+\theta^q)\leq \kappa(\theta)\leq C_2(1+\theta^q)
\end{equation*}
  for some constants $q>0$, and $C_1,C_2>0$.
  \end{abstract}

\maketitle

\section{Introduction}

Magnetohydrodynamics (MHD) studies the dynamics of conducting fluids in an magnetic field. The MHD finds its way in a very wide range of
physical objects, from liquid metals to cosmic plasmas, for example, see  \cite{C,JT64,KL,LL,PD}. The
governing equations of compressible  planar magnetohydrodynamic flows, which implies that the flows are uniform  in the transverse directions,  take the
following form:
\begin{align}
&\rho_t+(\rho u)_x=0,\label{aa1}\\
&(\rho u)_t+\left(\rho u^2+P+\frac{1}{2}|\bb|^2\right)_x=(\lambda
u_x)_x,\label{aa2}\\
&(\rho \w)_t+(\rho u \w-\bb)_x=(\mu \w_x)_x,\label{aa3}\\
&\bb_t+(u \bb-\w)_x=(\nu \bb_x)_x,\label{aa4}\\
&(\rho e)_t+(\rho u e)_x-(\kappa e_x)_x=\lambda
u_x^2+\mu|\w_x|^2+\nu|\bb_x|^2-Pu_x,\label{aa5}
\end{align}
where  $\rho\geq 0
$ denotes the density of the flow, $u\in\mathbb{R}$ the
longitudinal velocity, $\w\in\mathbb{R}^2$ the transverse velocity,
$\bb\in\mathbb{R}^2$ the transverse magnetic field, and $e$ the
internal energy, respectively.
Both the pressure $P$ and the internal
energy $e$ are generally related to the density and temperature of the flow according to the
equations of state: $P= P(\rho, \theta)$ and $e= e(\rho,\theta)$.
The parameters  $\lambda=\lambda (\rho, \theta) $ and $\mu=\mu(\rho, \theta)$ denote  the bulk
and the shear viscosity coefficients, respectively;  $\nu=\nu(\rho, \theta)$ is the magnetic diffusivity acting as a
magnetic diffusion coefficient of the magnetic field, and $\kappa = \kappa(\rho,\theta)$ is the heat conductivity.

The system (\ref{aa1})--(\ref{aa5}) are supplemented with the following initial and boundary conditions:
\begin{align}
 &(u,\w,\bb,\theta_x)|_{\partial\Omega} =0,\label{aa6}\\
 &(\rho,u,\w,\bb,\theta)|_{t=0}=\big(\rho_0(x),u_0(x),\w_0(x),\bb_0(x),\theta_0(x)\big),\label{aa7}
\end{align}
where $\partial\Omega=\{0,1\}$ denotes the boundary of the interval $\Omega:=(0,1)$. The conditions (\ref{aa6}) mean that the
boundary is non-slip and thermally insulated.

There have been a lot of studies on MHD by physicists and mathematicians
due to its physical importance, complexity, rich phenomena, and mathematical
challenges. Below we mention some mathematical results on existence theory of the compressible MHD equations, the interested readers can refer
\cite{C,JT64,KL,LL,PD} for complete discussions on  physical aspects.
We begin with the one-dimensional case. The existence and
uniqueness of local smooth solutions were proved firstly in \cite{VH}, while the existence
of global smooth solutions with small smooth initial data was shown in \cite{KO}.
The exponential stability of small smooth solutions was obtained in \cite{LZ,St}. In \cite{TH,HT},
Hoff and Tsyganov obtained the global existence and uniqueness of weak solutions with small initial energy.
Under the technical condition that $\kappa(\rho,\theta)$, depending the temperature $\theta$ only, i.e., $\kappa(\rho,\theta) \equiv \kappa(\theta)$,
satisfies
\begin{equation}
C_{1}(1+\theta^q)\leq \kappa(\theta)\leq C_2(1+\theta^q),\label{aa8}
\end{equation}
for  some constants  $q\geq2$ and $C_1,C_2>0$, Chen and Wang  \cite{CWa} proved the existence, uniqueness,
and Lipschitz continuous dependence of global strong solutions
to the system (\ref{aa1})--(\ref{aa5}) with large initial data satisfying
\begin{align*}%\label{aa8a}
 0<\inf_{x\in \Omega}\rho_0(x)\leq\rho_0(x)\leq\sup_{x\in \Omega}\rho_0(x)<\infty;~\rho_0,\,u_0,\,\w_0,\, \bb_0,\,\theta_0\in
H^1(\Omega), ~\inf_{x\in \Omega} \theta_0>0.
\end{align*}
The similar results are  obtained in
\cite{CWb,W} for real gas cases. Recently, Fan, Jiang and Nakamura  \cite{FJNa} obtained the global weak solutions
to the problem  (\ref{aa1})--(\ref{aa5}) when the initial data satisfying  the condition \eqref{aa8} with $q\geq 1$ and
\begin{align*}
  \rho_0^{-1},\rho_0\in L^\infty(\Omega); \ \  \rho_0,\,u_0,\,\w_0,\, \bb_0 \in L^2(\Omega); \ \  \theta_0\in L^1(\Omega), ~\inf_{x\in \Omega} \theta_0>0.
\end{align*}
Later they \cite{FJNb} obtained  the existence, the uniqueness and the Lipschitz continuous
dependence on the initial data of global weak solutions to  the problem  (\ref{aa1})--(\ref{aa7}) when the initial data
lie in the Lebesgue spaces. Recently, Hu and Ju \cite{HJ} considered the problem \eqref{aa1}--\eqref{aa7} under the assumption that the heat conductivity depends on temperature with \begin{equation}\label{PZZZ}
\kappa(\theta)=\theta^q, \quad  q>0,  
\end{equation} and obtained the existence and uniqueness of global strong solutions with large initial data. Their methods  are different from the one used here.
In fact, their arguments  were motivated by the ideas developed by Pan and Zhang \cite{PZ} where 
  the global existence of smooth solutions to
the 1-d compressible Navier-Stokes equations with arbitrarily large initial data under the condition \eqref{PZZZ} was obtained..

For the multi-dimensional
compressible MHD equations, there are also many mathematical
results on existence of solutions. As mentioned before, Vol'pert and Hudjaev \cite{VH} first
obtained   the local smooth solutions to the compressible MHD equations. Li, Su and
Wang \cite{LSW} obtained  the existence and uniqueness of local in
time strong solution with large initial data when the initial
density has an positive lower bound.  Fan and  Yu \cite{FYb}
obtained the strong solution to the compressible MHD equations with
vacuum. Kawashima \cite{K} obtained the smooth solutions for
two-dimensional compressible MHD equations when the initial data is
a small perturbation of given constant state. Umeda, Kawashima and
Shizuta \cite{UKS} obtained the decay of solutions to the linearized
MHD equations.  Li and Yu \cite{LY} obtained the optimal decay rate
of small smooth solutions.   In \cite{HW1, HW2}, Hu and Wang
obtained the global existence of weak solutions to the isentropic
compressible MHD equations and variational  solutions to the full
compressible MHD equations, see also \cite{FYa,DF,ZJX} for related
results.  Suen and  Hoff \cite{SH} obtained the  global low-energy
weak solutions of the  isentropic compressible MHD equations.
%Xu and Zhang \cite{XZ} obtained a blow-up criterion to the isentropic compressible MHD equations.
%We mention that
% the low Mach limit to the compressible MHD equations
%  is a  very important  topic, and  the interested reader
%  can refer  \cite{KM,JJL1,JJL2,JJL3,JJLX,HW3,K10,KT,NRT}
% and the references cited therein.

It should be pointed out that although there are many progress on
compressible MHD equations it is still an open question in obtaining the
global strong or smooth solutions to the full compressible MHD
equations with large initial data and possible vacuum even in the
one dimensional case, see \cite{HW2}.

In the present paper we study the global existence and uniqueness of
  large strong solutions to
 the planar compressible magnetohydrodynamic equations
 (\ref{aa1})-(\ref{aa5}) with large initial data and vacuum.
We focus on the perfect gas case:
$$ P(\rho,\theta):=R\rho\theta, \quad e:=C_V\theta,$$
where  $R\! >0$ is the
gas constant and $C_V\! >0$ is the heat capacity of the gas at constant
volume. We will consider the case that  the coefficients $\lambda$, $\mu$, and  $\nu$ are positive constants
 and the heat conductivity coefficient depends on the temperature $\theta$ only, i.e., $\kappa(\rho,\theta) \equiv \kappa(\theta)$.
%Since the positive physical constants $\lambda, \mu, \nu, R$, and $ C_V$  do not create essential mathematical
%difficulties in our analysis,  we normalize them to be one  for notational simplicity.

The main result in this paper reads as follows.
\begin{Theorem}\label{thaa1}
Let $\kappa\in C^2[0,\infty)$ satisfies the condition (\ref{aa8}) for some $q>0$. Suppose that the initial data
$(\rho_0,u_0,\w_0,\bb_0,\theta_0)$ satisfy $\rho_0\geq0,~\theta_0\geq0,~\rho_0\in H^2(\Omega),~u_0\in H_0^1(\Omega)\cap
H^2(\Omega), \w_0,\bb_0\in H_0^1(\Omega)\cap
H^2(\Omega),~\theta_0\in H^2(\Omega),~(\theta_0)_x|_{\partial\Omega}=0$ and the
following compatibility condition:
\begin{equation}
\left\{
\begin{array}{l}
\lambda (u_{0})_{xx}-\left(\rho_0\theta_0+\frac{1}{2}|\bb_0|^2\right)_x=\sqrt{\rho_0}\,\,g_1,\\
\mu ( \w_{0})_{xx}-(\bb_{0})_{x}=\sqrt{\rho_0}\,\,\mathbf{g}_2,\\
\big(\kk(\theta_0)(\theta_{0})_{x}\big)_x+
[(u_{0})_{x}]^2+ |(\w_{0})_{x}|^2+ |(\bb_{0})_{x}|^2=\sqrt{\rho_0}\,\,g_3,
\end{array}
\right.\label{aa9}
\end{equation}
for some $g_1, \mathbf{g}_2, g_3\in L^2(\Omega)$. Then, for any $T>0$, there exists a
unique global solution $(\rho,u,\w,\bb,\theta)$ to the problem
(\ref{aa1})--(\ref{aa7}) such that
\begin{align*}
&\rho\in L^\infty([0,T];H^2(\Omega)),~\rho_t\in L^\infty([0,T];H^1(\Omega)),\\
&(u,\w,\bb,\theta)\in L^\infty(0,T;H^2(\Omega))\cap L^2(0,T;H^3(\Omega)),\\
&(\sqrt\rho u_t,\sqrt\rho \w_t,\bb_t,\sqrt\rho\theta_t)\in L^\infty(0,T;L^2(\Omega)),\\
&(u_t,\w_t,\bb_t,\theta_t)\in L^2(0,T;H^1(\Omega)).
\end{align*}

\end{Theorem}

There are two ingredients in our result comparing with the previous results on one-dimensional
  MHD equations mentioned above. First, in our result the initial density may contain the vacuum provided
  that it satisfies the compatible conditions \eqref{aa9}. Next, a relaxed condition
  on the heat conductivity coefficient  is permitted. In fact, in \eqref{aa8}, we only need $q>0$ while
  in \cite{CWa,CWb,W} required $q\geq 2$ and in \cite{FJNa} with $q\geq 1$. We point out that the
  assumption $q>0$ paly a crucial role in our arguments, see Lemmas 2.3, 2.5, 2.7, and 2.9 below.

\begin{Remark}
As has been observed in \cite{CH06},  the lack of a positive lower bound of $\rho_0$ should be compensated
with some conditions on the initial data $(\rho_0,\w_0,b_0,\theta_0)$. %If  $(\rho,\w,b,\theta)$ is a sufficiently smooth
%solution of the system \eqref{aa1}--\eqref{aa5}, then letting $t\rightarrow 0$ in Eqs. \eqref{aa2}, \eqref{aa3} and \eqref{aa5},
%we readily derive a natural condition stronger than \eqref{aa9}.  But it turns out that the compatibility \eqref{aa9}
% is sufficient to prove the existence and uniqueness
%of  strong solutions.
 Roughly speaking, the compatibility condition \eqref{aa9} is equivalent to the $L^2$-integrability of
$\sqrt{\rho} u_t, \sqrt{\rho}\w_t,$ and $ \sqrt{\rho}\theta_t$
 at $t = 0$, as can be shown formally by letting $ t \rightarrow 0$ in  \eqref{aa2}, \eqref{aa3}, and \eqref{aa5}. Hence the condition \eqref{aa9} plays
a key role in the  estimates of   $(\sqrt{\rho} u_t, \sqrt{\rho}\w_t, \bb_t, \sqrt{\rho}\theta_t)$.
This was observed and justified rigorously  in \cite{CH06} for viscous polytropic fluids. Note that
the  condition \eqref{aa9}  is satisfied automatically for all initial data $(\rho_0,u_0,\w_0,\bb_0,\theta_0)$  with the
regularity presented in Theorem \ref{thaa1}  whenever $\rho_0$ is bounded away from zero.
\end{Remark}

\begin{Remark}
 It is possible  to extend our results to the one-dimensional compressible MHD system with more
 general state of equations:
 \begin{align*}
   P=\rho^2\frac{\partial e}{\partial \rho}+\theta \frac{\partial P}{\partial \theta}
 \end{align*}
 with some additional assumptions.
\end{Remark}

\begin{Remark}
When there is no vacuum initially, we can improve the results in \cite{FJNa} to  the case $q>0$
 by applying the arguments developed here.
\end{Remark}

 We remark that
when taking $\w=\bb=0$  the system (\ref{aa1})--(\ref{aa7}) reduces to the
well-known one-dimensional full Navier-Stokes equations and there are a lot of studies on this system.
In the case of that the initial density is bounded away from zero, Kazhikhov  and Shelukhin
\cite{KS} first obtained the global smooth solutions for large initial data three decades ago,see also \cite{AZa,ZAa,ZAb,Kaw} for  different extensions. Recently, Huang and Li \cite{HL} obtain the
global smooth solutions to the full Navier-Stokes system with possible vacuum
and large oscillations provided that the total initial energy is sufficient small.
 Wen and Zhu obtained the global smooth solutions to the one-dimensional full Navier-Stokes system
  \cite{WZa} and symmetric higher dimensional full  Navier-Stokes system  \cite{WZb} with large initial data.
  For the variational or weak solutions to the full Navier-Stokes system, see \cite{BD07, Fei04}.

   We give a few words on the strategy of the proof.  Since the initial data may contains the vacuum we first
 construct  the regularized initial density $\rho_{0 \delta}(x)=\rho_0+\delta$  for any $\delta>0$.
 Next, for each fixed $\delta$, we can obtain the local and uniqueness  existence of strong solutions.
  Third, we establish   sufficient  \emph{a priori} estimates uniformly with $\delta$.
 Combining the local existence result,  and the uniformly  a priori  estimates,
 we obtain the desired  global existence result by taking the limiting as $\delta \rightarrow 0^+$  and
 applying standard continuity argument.
We remark that the key point in the whole proof is to obtain uniformly  a priori  estimates  where
some ideas developed in \cite{WZa,WZb} are adapted.  Comparing with \cite{WZa}, the main
additional difficulties are due to the presence of the magnetic field and its interaction with the
hydrodynamic motion of the flow of large oscillation.
We shall deal with the  terms involving the magnetic field very carefully, see especially Lemmas
\ref{lebb5}-\ref{lebb7} below.

%
%
%Very recently, Wen-Zhu
%\cite{WZa} proved the global classical large solutions with vacuum.
%The aim of this paper is to generalize the results in \cite{WZa} to
%the full MHD system (\ref{aa1})--(\ref{aa7}). Here we should note
%that the key difference between our paper and \cite{WZa} lies in the
%estimate of the upper bound of the density (Lemma 2.2). For the 1-D
%full Navier-Stokes equations, it is easy to bound the density
%without assuming (\ref{aa8}) for the heat conductivity. For the full
%MHD system, we must assume (\ref{aa8}), so it is an open problem to
%prove the density estimate when the heat conductivity is a positive
%constant.

Before leaving this introduction we recall the following auxiliary inequalities.

\begin{Lemma}[\cite{Fei04, WZa}]\label{pona}
Let $  \Omega=(0,1)$ be an interval in $\mathbb{R}^1$.

(i). Assume that
$\rho$ is  a non-negative function  satisfying
\begin{equation*}
 0<M\leq\int_\Omega \rho\,dx\leq K,
\end{equation*}
for two  constants $M$ and $K$.
 Then, for any $v\in H^1(\Omega)$, it holds
\begin{equation*}
 \|v\|_{L^\infty(\Omega)}\leq \frac{K}{M}\|v_x\|_{L^2(\Omega)}+\frac{1}{M}\left|\int_\Omega \rho vdx\right|,
\end{equation*}

  (ii) Assume further that  $v$ satisfies
$$
\|\rho v\|_{L^1(\Omega)}\le \bar{C}.
$$
Then for any $r>0$, there exists a positive constant $  C=C(M, K, r,
\bar{C})$ such that
\begin{equation*}
   \|v^r\|_{L^\infty(\Omega)}\leq
C\|(v^r)_x\|_{L^2(\Omega)}+  C.
\end{equation*}
\end{Lemma}

%
%\begin{Lemma}[\cite{WZa}]\label{ponb}
%For any $v\in H^1_0(\Omega)$, we have
%$$
%\|v\|_{L^\infty(\Omega)}\le \|v_x\|_{L^1}\leq \|v_x\|_{L^2}.
%$$
%\end{Lemma}
%
  \medskip
%%%%%%%%%%%-----------------%%%%%%%%%%%%%%%%%%%%%%%%%%%%%%%%%%%%%%%%%%%%%%%%%%

\section{Proof of Theorem \ref{thaa1}}

In this section  we will prove Theorem \ref{thaa1} by consider  the initial density
$\rho_{0\delta} =\rho_0+\delta$,  as mentioned before, to get a sequence of approximate
solutions to (\ref{aa1})--(\ref{aa7}), then taking
$\delta\rightarrow0^+$ after making some   a priori  estimates uniformly for
$\delta$. Since the proof of the local existence and  uniqueness of strong solutions to the
approximate problem is now standard
\cite{CH06,FYb}, thus we only need
to establish the uniform estimates.

Below we still use $(\rho,u,\w,\bb,\theta)$ to denote the smooth solutions of
approximate problem  to (\ref{aa1})--(\ref{aa7}).
We shall
 denote $Q_T:=\Omega\times [0,T]$ with $T>0$ and  omit the spatial domain $\Omega$ in integrals for convenience.
We use $C$  to denote
the constants which are independent of  $\delta$, but possible depending on $T$, and may change from line to line.

To begin with the proof, we notice that the total mass and energy in the system
(\ref{aa1})--(\ref{aa5}) are conserved. In fact, by rewriting (\ref{aa1})--(\ref{aa5}) one has
\begin{align}
&\mathcal{E}_t+\left[u\left(\mathcal{E}+P+\frac{1}{2}|\bb|^2\right)-\w\cdot
\bb\right]_x=(\lambda uu_x+ \mu \w\cdot\w_x+ \nu
 \bb\cdot\bb_x+\kappa(\theta)\theta_x)_x,\label{bb1}\\
&(\rho \mathcal{S})_t+(\rho
u\mathcal{S})_x-\left(\frac{\kappa(\theta)}{\theta}\theta_x\right)_x=\frac{
\lambda u_x^2+\mu |\w_x|^2+\nu |\bb_x|^2}{\theta}+\frac{\kappa(\theta)\theta_x^2}{\theta^2},\label{bb2}
\end{align}
where $\mathcal{E}$ and $\mathcal{S}$ are the total energy and the entropy,
respectively,
$$\mathcal{E}:=\rho\Big(\theta+\frac{1}{2}(u^2+|\w|^2)\Big)+\frac{1}{2}|\bb|^2,\quad \mathcal{S}:=\ln\theta-\ln\rho.$$
Integrating (\ref{aa1}), (\ref{bb1}) and (\ref{bb2}) over
$Q_T$, we have
\begin{Lemma}\label{lebb1}
\begin{align*}
&\int \rho(x,t)dx=\int \rho_0(x)dx, \qquad \int
\mathcal{E}(x,t)dx=\int  \mathcal{E}(x,t=0)dx,\\
&\int (\rho\ln\rho+\rho|\ln\theta|)(x,t)dx+\iint_{Q_T} \left(\frac{
u_x^2+|\w_x|^2+|\bb_x|^2}{\theta}+\frac{\kappa(\theta)\theta_x^2}{\theta^2}\right)dx
dt\leq C.
\end{align*}
\end{Lemma}

\begin{Lemma}\label{lebb2}
\begin{equation}
0\leq \rho(x,t)\leq C, \ \ (x,t)\in \overline{Q}_T.\label{bb3}
\end{equation}
\end{Lemma}
\begin{proof}
We need only to estimate the upper bound.   The proof is similar to that in \cite{FJNa}, here we present it for completeness.
  From \eqref{aa1} and \eqref{aa2}, we have
\begin{align*}
  (\rho u)_t=\tilde{P}_x, \quad \tilde{P}_x:= \lambda u_x-\rho u^2-P-\frac12 |\bb|^2.
\end{align*}
Denote
\begin{align*}
   \phi:=\int^t_0 \tilde{P}(x,\tau)d\tau+\int^x_0\rho_0(\xi)u_0(\xi)d\xi,
\end{align*}
we have
\begin{align}
   \label{fjnaa}
   \phi_x=\rho u, \quad \phi_t=\tilde{P}, \quad \phi_x|_{\partial \Omega}=0,
   \quad \phi|_{t=0}=\int^x_0\rho_0(\xi)u_0(\xi)d\xi.
\end{align}
By virtue of Lemma \ref{lebb1} and   Cauchy inequality, it holds
\begin{align*}
   \|\phi_x\|_{L^\infty(0,T;L^1(\Omega))}\leq C, \quad \Big|\int \phi dx\Big| \leq C.
\end{align*}
Hence
\begin{align}\label{fjnab}
  \|\phi\|_{L^\infty(0,T;L^\infty(\Omega))}\leq C.
\end{align}

Now, denoting $F:=e^\phi$ and using \eqref{fjnaa}, we have after a straightforward calculation that
\begin{align*}
  D_t(\rho F):=\partial_t(\rho F)+u \partial_x(\rho F)=-\Big(P+\frac12|\bb|^2\Big)\rho F\leq 0,
\end{align*}
which together with \eqref{fjnab} implies \eqref{bb3} immediately.
\end{proof}

\begin{Lemma}\label{lebb3}
Let $0< \alpha<\min\{1, q\}$ be any given
 constant, then it holds
\begin{align}
&\iint_{Q_T}\left(\frac{\lambda u_x^2+\mu |\w_x|^2+\nu |\bb_x|^2}{\theta^\alpha}+
\frac{(1+\theta^q)\theta_x^2}{\theta^{1+\alpha}}\right)dxdt\leq
C,\label{bb4}\\
&\int_0^T\|\theta\|_{L^\infty}^{q-\alpha+1}dt\leq C.\label{bb5}
\end{align}
\end{Lemma}

\begin{proof}
The proof is similar to that given in \cite{WZb} for symmetric Navier-Stokes equations, we present it here for completeness.
Multiplying \eqref{aa5} by $\theta^{-\alpha}$ and integrating the result over $Q_T$, we have
\begin{align}\label{bb55}
&\iint_{Q_T}\frac{\lambda u_x^2+\mu |\w_x|^2+\nu |\bb_x|^2}{\theta^\alpha}dxdt+
\alpha\iint_{Q_T}\frac{\kappa(\theta)\theta_x^2}{\theta^{1+\alpha}}dxdt\nonumber\\
&\qquad\qquad\qquad=\iint_{Q_T}\{(\rho\theta)_t+(\rho u \theta)_x+Pu_x\}\theta^{-\alpha} dxdt.
\end{align}
Using \eqref{aa1}, Lemma \ref{lebb1}, and Young's inequality, the first two terms on the right-hand side of \eqref{bb55} can be bounded as
\begin{align}\label{bb56}
  \iint_{Q_T}\{(\rho\theta)_t+(\rho u \theta)_x\}\theta^{-\alpha} dxdt
 =&   \iint_{Q_T}(\rho\theta_t\theta^{-\alpha}+\theta_x\rho u\theta^{-\alpha})dxdt\nonumber\\
 =& \frac{1}{1-\alpha}\iint_{Q_T}[(\rho\theta^{1-\alpha})_t+(\rho u \theta^{1-\alpha})_x]dxdt\nonumber\\
 =&  \frac{1}{1-\alpha}\left[\int \rho\theta^{1-\alpha}dx-\int \rho_0\theta^{1-\alpha}_0dx\right]\nonumber\\
 \leq & \frac{1}{1-\alpha}\int \rho\theta^{1-\alpha}dx\nonumber\\
 \leq & \frac{C}{1-\alpha}\left[\int \rho\theta dx+\int \rho dx\right]\nonumber\\
 \leq & C.
 % =& \int_\Omega\rho\int^\theta_0\frac{1}{\xi^\alpha}d\xi dxdt
%  -\int_\Omega\rho_0\int^{\theta_0}_0\frac{1}{\xi^\alpha}d\xi dx\nonumber\\
% = & \int_\Omega\frac{1}{1-\alpha}\rho\theta^{1-\alpha}dxdt -
% \int_\Omega\frac{1}{1-\alpha}\rho_0\theta_0^{1-\alpha}dx
% \nonumber\\
% \leq & C+C\int^T_0\|\theta\|_{L^\infty}^{1-\alpha}dt.
\end{align}
By Cauchy inequality, Lemmas  \ref{lebb1} and \ref{lebb2}, we have
\begin{align}
\iint_{Q_T}  Pu_x \theta^{-\alpha} dxdt \leq & \frac{\lambda}{2}\iint_{Q_T}\frac{u_x^2}{\theta^\alpha}dxdt
+C\iint_{Q_T} \rho^2\theta^{2-\alpha}dxdt\nonumber\\
\leq &  \frac{\lambda}{2}\iint_{Q_T}\frac{u_x^2}{\theta^\alpha}dxdt
+C \int^T_0\|\theta\|_{L^\infty}^{1-\alpha}dt.
\end{align}
Noticing $0<\alpha<\min\{1,q\}$,  the H\"{o}lder inequality and Lemma \ref{pona} imply that
\begin{align}\label{bb57}
 C \int^T_0\|\theta\|_{L^\infty}^{1-\alpha}dt\leq & C+\int^T_0 \|\theta^{-\alpha}\theta_x\|_{L^2}dt \nonumber\\
 \leq &  C +C \int^T_0\left(\int \frac{\theta^2_x\theta^{1-\alpha}}{\theta^{1+\alpha}}dx\right)^{1/2}dt\nonumber\\
 \leq & C +\frac{\alpha}{2} \iint_{Q_T}\frac{\kappa(\theta)\theta_x^2}{\theta^{1+\alpha}}dxdt
\end{align}
for $ q\geq 1-\alpha$, while
\begin{align}\label{bb58}
 C \int^T_0\|\theta\|_{L^\infty}^{1-\alpha}dt\leq & C+\int^T_0 \|\theta^{-\alpha}\theta_x\|_{L^2}dt \nonumber\\
 \leq  &  C +C \int^T_0\left(\int \frac{\theta^q\theta^2_x}{\theta^{1+\alpha}} \theta^{1-\alpha-q} dx\right)^{1/2}dt\nonumber\\
 \leq & C +\frac{\alpha}{2} \iint_{Q_T}\frac{\kappa(\theta)\theta_x^2}{\theta^{1+\alpha}}dxdt
+\tilde{C} \int^T_0\|\theta\|_{L^\infty}^{1-\alpha-q}dt\nonumber\\
\leq &  C +\frac{\alpha}{2} \iint_{Q_T}\frac{\kappa(\theta)\theta_x^2}{\theta^{1+\alpha}}dxdt
+\frac{C}{2} \int^T_0\|\theta\|_{L^\infty}^{1-\alpha}dt
\end{align}
for $ 0<q<1-\alpha$.

Putting \eqref{bb56}-\eqref{bb58} into \eqref{bb55} implies \eqref{bb4}.  By \eqref{bb57}, \eqref{bb58}, and \eqref{bb4},
we can easily get \eqref{bb5}.
\end{proof}

\begin{Lemma}\label{lebb4}
\begin{equation}
\iint_{Q_T}(\lambda u_x^2+\mu |\w_x|^2+\nu |\bb_x|^2)dx dt\leq C.\label{bb6}
\end{equation}
\end{Lemma}
\begin{proof}
  Multiplying (\ref{aa2}) by $u$, using (\ref{aa1}), and integrating the result over
$\Omega$,  we see that
\begin{equation}
\frac{1}{2}\frac{d}{dt}\int \rho u^2 dx+\int \lambda u_x^2~dx=\int P u_x\,dx-\int
u \bb\cdot\bb_x dx.\label{bb7}
\end{equation}
 Similar, multiplying  (\ref{aa3}) by $\w$, using (\ref{aa1}), and integrating the result over
$\Omega$, we obtain
\begin{equation}
\frac{1}{2}\frac{d}{dt}\int\rho |\w|^2dx+\int \mu |\w_x|^2dx=\int \bb_x
\cdot \w dx=-\int \bb\cdot \w_x dx.\label{bb8}
\end{equation}
Multiplying (\ref{aa4}) by $\bb$ and then integrating them    over
$\Omega$,   we infer that
\begin{equation}
\frac{1}{2}\frac{d}{dt}\int |\bb|^2dx+\int \nu |\bb_x|^2dx=\int \bb \cdot\w_x
dx+\int u \bb\cdot\bb_x dx.\label{bb9}
\end{equation}

Summing up (\ref{bb7}), (\ref{bb8}) and (\ref{bb9}), using
(\ref{bb3}) and (\ref{bb5}), we get
\begin{align}\label{bb91}
&\frac{1}{2}\frac{d}{dt}\int(\rho u^2+\rho
|\w|^2+|\bb|^2)dx+\int(\lambda u_x^2+\mu |\w_x|^2+\nu |\bb_x|^2)dx=\int P u_x\,dx.
\end{align}
By Young inequality, Lemmas \ref{lebb1} and \ref{lebb2}, the estimate \eqref{bb5}, and the condition that $0<\alpha <\min\{1,q\}$,
we have
 \begin{align*}
\iint_{Q_T} P u_x\,dxdt
 \leq&C \iint_{Q_T} \rho^2 \theta^2dxdt +\frac12\int^T_0 \|u_x\|_{L^2}^2dt \nonumber\\
\leq &
C\int^T_0\|\theta\|_{L^\infty}dt +\frac12\int^T_0 \|u_x\|_{L^2}^2dt\nonumber\\
\leq &C+ C\int^T_0\|\theta\|^{q-\alpha+1}_{L^\infty}dt +\frac12\int^T_0 \|u_x\|_{L^2}^2dt\nonumber\\
\leq & C+  \frac12\int^T_0 \|u_x\|_{L^2}^2dt.
\end{align*}
Integrating \eqref{bb91} over $[0,T]$ and applying the above inequality give
  (\ref{bb6}).
\end{proof}

\begin{Lemma}\label{lebb5}
\begin{align}
&\int(\lambda u_x^2+\mu |\w_x|^2+\nu |\bb_x|^2+P^2+\rho\theta^{q+2})dx\nonumber\\
 &+\iint_{Q_T}(\rho u_t^2+\rho
|\w_t|^2+|\bb_t|^2+\nu |\bb_{xx}|^2+\kappa^2(\theta)\theta_x^2)dx dt\leq
C.\label{bb10}
\end{align}
\end{Lemma}
\begin{proof}
 Multiplying (\ref{aa2}) by $u_t$, and then integrating them    over
$\Omega$,   we infer that
\begin{align}
 &\frac{1}{2}\frac{d}{dt}\int u_x^2dx+\int\rho u_t^2dx+\int\rho
uu_xu_t dx\nonumber\\
 =& -\int\left(P+\frac{1}{2}|\bb|^2\right)_xu_t dx\nonumber\\
 =&\int\left(P+\frac{1}{2}|\bb|^2\right)u_{xt}dx\nonumber\\
 =& \frac{d}{dt}
\int\left(P+\frac{1}{2}|\bb|^2\right)u_xdx-\int\left(P+\frac{1}{2}|\bb|^2\right)_t
u_x dx\nonumber\\
 =&\frac{d}{dt}\int\left(P+\frac{1}{2}|\bb|^2\right)u_x
dx-\int\left(P+\frac{1}{2}|\bb|^2\right)_t\left(u_x-P-\frac{1}{2}|\bb|^2\right)dx\nonumber\\
&-\frac{1}{2}\frac{d}{dt}\int\left(P+\frac{1}{2}|\bb|^2\right)^2dx.\label{bb11}
\end{align}

First, by Cauchy inequality and Poincar\'{e} inequality, it is easy to find that
\begin{align}
\left|\int\rho uu_xu_t dx\right| \leq&\frac{1}{16}\int\rho
u_t^2dx+C\int\rho u^2u_x^2dx\nonumber\\
 \leq&\frac{1}{16}\int\rho
u_t^2dx+C\|\rho\|_{L^\infty}\|u\|^2_{L^\infty}\int u_x^2dx\nonumber\\
 \leq&C+\frac{1}{16}\int\rho u_t^2dx+C\left(\int
u_x^2dx\right)^2.\label{bb12}
\end{align}
From  (\ref{aa4}) and (\ref{aa5}), we have
\begin{align}
&\left(\frac{1}{2}|\bb|^2\right)_t+\bb\cdot(u
\bb-\w)_x=(\bb\cdot\bb_x)_x-\nu |\bb_x|^2,\label{bb13}\\
&P_t+(uP-\kappa(\theta)\theta_x)_x=\lambda u_x^2+\mu |\w_x|^2+\nu |\bb_x|^2-Pu_x.\label{bb14}
\end{align}
Using (\ref{bb13}), (\ref{bb14}), (\ref{aa2}),  (\ref{bb3}), and Lemma \ref{lebb2},  we
obtain
\begin{align}
&-\int\left(P+\frac{1}{2}|\bb|^2\right)_t\left(\lambda u_x-P-\frac{1}{2}|\bb|^2\right)dx\nonumber\\
 =&-\int\left[\lambda u_x^2+\mu |\w_x|^2-Pu_x-\bb\cdot\left(u
\bb-\w\right)_x\right]\left(\lambda u_x-P-\frac{1}{2}|\bb|^2\right)dx\nonumber\\
&+\int\left(\kappa(\theta)\theta_x-uP+\bb\cdot\bb_x\right)\left(\lambda u_x-P-\frac{1}{2}|\bb|^2\right)_x
dx\nonumber\\
 =&-\int\left[\lambda u_x^2+\mu |\w_x|^2-Pu_x-\bb\cdot(u
\bb-\w)_x\right]\left(\lambda u_x-P-\frac{1}{2}|\bb|^2\right)dx\nonumber\\
&+\int\left(\kappa(\theta)\theta_x-uP+\bb\cdot\bb_x\right)\left(\rho u_t+\rho uu_x\right)dx\nonumber\\
 \leq&\int\Big |\lambda u_x^2+\mu |\w_x|^2-Pu_x-\bb\cdot\left(u
\bb-\w\right)_x\Big|dx\left\|\lambda u_x-P-\frac{1}{2}|\bb|^2\right\|_{L^\infty}\nonumber\\
&+\left\|\kappa(\theta)\theta_x-uP+\bb\cdot\bb_x\right\|_{L^2}\left(\|\sqrt\rho
u_t\|_{L^2}\| \rho\|_{L^1}^{1/2}+\|\rho\|_{L^\infty}\|u\|_{L^\infty}\|u_x\|_{L^2}\right)\nonumber\\
 \leq& C\left(\left\|u_x-P-\frac{1}{2}|\bb|^2\right\|_{L^1} +\left\|\left(\lambda u_x-P-\frac{1}{2}|\bb|^2\right)_x\right\|_{L^1}\right)
 \nonumber\\
&\times  \int\left(\lambda u_x^2+\mu |\w_x|^2+P^2+|\bb|^2+|\bb|^4+u^2+|\bb_x|^2\right)dx\nonumber\\
&   +C\left\|\kappa(\theta)\theta_x-u
P+\bb\cdot\bb_x\right\|_{L^2}\left(\|\sqrt\rho u_t\|_{L^2}+C\|u_x\|_{L^2}^2\right)\nonumber\\
 \leq& C\int\left(\lambda u_x^2+\mu |\w_x|^2+P^2+|\bb_x|^2\right)dx\Big\{1+\|u_x\|_{L^2}+\|\rho u_t+\rho
uu_x\|_{L^1}\Big\}\nonumber\\
&+C\left\|\kappa(\theta)\theta_x-uP+\bb\cdot\bb_x\right\|_{L^2}\left(\|\sqrt\rho
u_t\|_{L^2}+C\|u_x\|_{L^2}^2\right)\nonumber\\
 \leq& C\Big\{1+\|u_x\|_{L^2}^4+\|\w_x\|_{L^2}^4+\|\bb_x\|_{L^2}^4+\|\kappa(\theta)\theta_x\|_{L^2}^2\Big\}
+\frac{1}{16}\left\|\sqrt\rho u_t\right\|_{L^2}^2.\label{bb15}
\end{align}
Multiplying (\ref{aa3}) by $\w_t$, using (\ref{bb3}), and integrating the result over
$\Omega$,  we derive
\begin{align}
&\frac{\mu}{2}\frac{d}{dt}\int |\w_x|^2dx+\int\rho |\w_t|^2dx\nonumber\\
= & \int
\bb_x\cdot\w_t
dx-\int\rho u \w_x\cdot\w_tdx\nonumber\\
 =&-\int \bb\cdot \w_{xt}dt-\int\rho u \w_x\cdot\w_tdx\nonumber\\
 =&-\frac{d}{dt}\int \bb\cdot  \w_x dx+\int \bb_t\cdot  \w_x dx-\int\rho u \w_x\cdot  \w_t
dx\nonumber\\
 \leq&-\frac{d}{dt}\int \bb\cdot  \w_x dx+\frac{1}{16}\int |\bb_t|^2dx+C\int
|\w_x|^2dx\nonumber\\
&+\frac{1}{16}\int\rho
|\w_t|^2dx+C\|u_x\|_{L^2}^4+C\|\w_x\|_{L^2}^4.\label{bb16}
\end{align}
Multiplying  (\ref{aa4}) by $\bb_t-\nu \bb_{xx}$,   integrating the result over
$\Omega$, and using Cauchy inequality,   we have
\begin{align}
&\frac{d}{dt}\int
|\bb_x|^2dx+\int(|\bb_t|^2+\nu|\bb_{xx}|^2)dx\nonumber\\
= & \int(\w-u
\bb)_x\cdot (\bb_t-\nu\bb_{xx})dx\nonumber\\
 \leq&C\int
|\w_x|^2dx+C\|u_x\|_{L^2}^4+C\|\bb_x\|_{L^2}^4+\frac{1}{16}\int(|\bb_t|^2+\nu|\bb_{xx}|^2)dx.\label{bb17}
\end{align}
Multiplying  (\ref{aa5}) by $\theta^{q+1}$, using (\ref{aa8}), and integrating the result over
$\Omega$,    we find
that
\begin{align}
&\frac{d}{dt}\int\rho\theta^{q+2}dx+C\int
\kappa^2\theta_x^2dx\nonumber\\
 \leq&C\int(\lambda u_x^2+\mu|\w_x|^2+\nu|\bb_x|^2+P^2)\theta^{q+1}dx\nonumber\\
 \leq&C\int(\lambda u_x^2+\mu|\w_x|^2+\nu|\bb_x|^2+P^2)dx\|\theta^{q+1}\|_{L^\infty}\nonumber\\
 \leq&C\int(\lambda u_x^2+\mu|\w_x|^2+\nu|\bb_x|^2+P^2)dx(\|\kappa\theta_x\|_{L^2}+1)\nonumber\\
 \leq&\frac{C}{16}\|\kappa\theta_x\|_{L^2}^2+C\left(\int(\lambda u_x^2+\mu|\w_x|^2+\nu|\bb_x|^2+P^2)dx\right)^2+C,\label{bb18}
\end{align}
where we used the estimate:
\begin{equation}
\|\theta\|_{L^\infty}^{q+1}\leq
C(1+\|\kappa(\theta)\theta_x\|_{L^2}) \label{bb19}
\end{equation}
which can be derived from  Lemma \ref{pona} and \eqref{aa8}.

Combining (\ref{bb11}), (\ref{bb12}), (\ref{bb15}), (\ref{bb16}), and
(\ref{bb17})  with (\ref{bb18}), we arrive at (\ref{bb10}).
\end{proof}

\begin{Lemma}\label{lebb6}
\begin{equation}
\int(\rho_x^2+\rho_t^2)dx+\iint_{Q_T}(u_{xx}^2+|\w_{xx}|^2)dx dt\leq C.\label{bb20}
\end{equation}
\end{Lemma}

\begin{proof}
Applying the operator $\partial_x$ to (\ref{aa1}) gives
$$\rho_{xt}+\rho_{xx}u+2\rho_xu_x+\rho u_{xx}=0.$$
Multiplying  the above
equation by $2\rho_x$,  integrating the result over
$\Omega$,   and  using (\ref{bb3}) and (\ref{bb10}), we find
that
\begin{align}
 \frac{d}{dt}\int\rho_x^2dx \nonumber
=& -3\int\rho_x^2u_xdx-2\int\rho\rho_xu_{xx}dx\nonumber\\
 =&-3\int\rho_x^2\left(\lambda u_x-P-\frac{1}{2}|\bb|^2\right)dx-3\int\rho_x^2
\left(p+\frac{1}{2}|\bb|^2\right)dx\nonumber\\
& -2\int\rho\rho_xu_{xx}dx\nonumber\\
 \leq&-3\int\rho_x^2\left(\lambda u_x-P-\frac{1}{2}|\bb|^2\right)dx\nonumber\\
 & -2\int\rho\rho_xu_{xx}dx\nonumber\\
 \leq&3\left\|\lambda u_x-P-\frac{1}{2}|\bb|^2\right\|_{L^\infty}\int\rho_x^2dx
+C\|\rho\|_{L^\infty}\|\rho_x\|_{L^2}\|u_{xx}\|_{L^2}\nonumber\\
 \leq&C\left(\left\|\lambda u_x-P-\frac{1}{2}|\bb|^2\right\|_{L^2}+
\left\|\left(\lambda u_x-P-\frac{1}{2}|\bb|^2\right)_x\right\|_{L^2}\right)
\int\rho_x^2dx\nonumber\\
&+C\|\rho_x\|_{L^2}\|u_{xx}\|_{L^2}\nonumber\\
 \leq&C(1+\|\sqrt\rho
u_t\|_{L^2})\|\rho_x\|_{L^2}^2+C\|u_{xx}\|_{L^2}^2.\label{bb21}
\end{align}
On the other hand, it follows from (\ref{aa2}) that
\begin{align}
\|u_{xx}\|_{L^2} \leq&C(\|\sqrt\rho u_t\|_{L^2}+\|\rho
uu_x\|_{L^2}+\|P_x\|_{L^2}+\|\bb\cdot \bb_x\|_{L^2})\nonumber\\
 \leq&C(1+\|\sqrt\rho
u_t\|_{L^2}+\|u_x\|_{L^2}^2+\|\theta_x\|_{L^2}+\|\rho_x\|_{L^2}\|\theta\|_{L^\infty})\nonumber\\
 \leq&C(1+\|\sqrt\rho
u_t\|_{L^2}+\|u_x\|_{L^2}^2+\|\kappa(\theta)\theta_x\|_{L^2}+\|\rho_x\|_{L^2}\|\theta\|_{L^\infty}).\label{bb22}
\end{align}
Inserting (\ref{bb22}) into (\ref{bb21}), using Lemmas \ref{bb3} and \ref{lebb5}, \eqref{bb19},
and the Gronwall inequality, we have
\begin{equation}
\|\rho_x\|_{L^\infty(0,T;L^2(\Omega))}\leq C.\label{bb23}
\end{equation}
It follows from (\ref{aa1}), (\ref{bb3}), (\ref{bb10}) and
(\ref{bb23}) that $$\|\rho_t\|_{L^\infty(0,T;L^2(\Omega))}\leq C.$$ Thus
(\ref{bb22}) yields $$\|u_{xx}\|_{L^2(0,T;L^2(\Omega))}\leq C.$$ It follows
from (\ref{aa3}), (\ref{bb3}), and (\ref{bb10}) that
$$\|\w_{xx}\|_{L^2(0,T;L^2(\Omega))}\leq C.$$
\end{proof}

\begin{Lemma}\label{lebb7}
\begin{align}
& \int(\rho u_t^2+\rho
|\w_t|^2+|\bb_t|^2+u_{xx}^2+|\w_{xx}|^2+|\bb_{xx}|^2+\kappa^2(\theta)\theta_x^2)dx\nonumber\\
&+\iint_{Q_T}(\lambda u_{xt}^2+\mu |\w_{xt}|^2+\nu |\bb_{xt}|^2+\rho\theta_t^2+\theta_{xx}^2)dx
dt\leq C.\label{bb24}
\end{align}
\end{Lemma}
\begin{proof} Applying $\partial_t$ to (\ref{aa2}), we see that
$$
\rho u_{tt}+\rho uu_{xt}-\lambda u_{xxt}=-\left(P+\frac{1}{2}|\bb|^2\right)_{xt}-\rho_tu_t-\rho_tuu_x-\rho
u_tu_x.
$$
 Multiplying  the above equation by $u_t$, integrating them    over
$\Omega$, and   using (\ref{aa1}),
Lemmas \ref{lebb5} and  \ref{lebb6}, and Cauchy inequality,  we obtain
\begin{align}
&\frac{1}{2}\frac{d}{dt}\int\rho u_t^2dx+\int \lambda u_{xt}^2dx\nonumber\\
 =&\int\left(P+\frac{1}{2}|\bb|^2\right)_t u_{xt}dx-2\int\rho uu_t
u_{xt}dx-\int\rho_t uu_x u_t dx-\int\rho u_t^2u_x dx\nonumber\\
 \leq&\int(\rho\theta_t+\theta\rho_t+\bb\cdot \bb_t)u_{xt}dx+2\|\sqrt\rho
u_t\|_{L^2}\|\sqrt\rho\|_{L^\infty}\|u\|_{L^\infty}\|u_{xt}\|_{L^2}\nonumber\\
&+\|\rho_t\|_{L^2}\|u\|_{L^\infty}\|u_x\|_{L^2}\|u_t\|_{L^\infty}
+\|u_x\|_{L^\infty}\int\rho u_t^2dx\nonumber\\
 \leq&\epsilon_1\int
u_{xt}^2dx+C\int\rho\theta_t^2dx+\|\theta\|_{L^\infty}^2+\|\bb_t\|_{L^2}^2\nonumber\\
&+C\int\rho u_t^2dx+\|u_{xx}\|_{L^2}\int\rho u_t^2dx+C \label{bb25}
\end{align}
for any $0<\epsilon_1<1$.

Multiplying (\ref{aa5}) by $\kappa(\theta)\theta_t=\left(\int_0^\theta
k(\xi)d\xi\right)_t$, integrating them    over
$\Omega$, and using (\ref{aa1}) and  Lemmas \ref{lebb5} and
\ref{lebb6}, we deduce that
\begin{align}
&\frac{1}{2}\frac{d}{dt}\int \kappa^2(\theta)\theta_x^2dx+\int\rho
\kappa(\theta)\theta_t^2dx\nonumber\\
 =&-\int\rho u\theta_x\kappa(\theta)\theta_t dx-\int\rho\theta
u_x\kappa(\theta)\theta_t
dx\nonumber\\
&+\int(\lambda u_x^2+\mu |\w_x|^2+\lambda |\bb_x|^2)\left(\int_0^\theta
\kappa(\xi)d\xi\right)_tdx\nonumber\\
 \leq&\epsilon_2\int\rho \kappa(\theta)\theta_t^2dx+C\int\rho
u^2\kappa(\theta) \theta_x^2dx+C\int\rho\theta^2\kappa (\theta)u_x^2
dx\nonumber\\
&+\frac{d}{dt}\int(\lambda u_x^2+\mu |\w_x|^2+\nu|\bb_x|^2)\int_0^\theta
\kappa(\xi)d\xi
dx\nonumber\\
& -2\int(u_xu_{xt}+\w_x\cdot \w_{xt}+\bb_x\cdot \bb_{xt})\int_0^\theta
\kappa(\xi)d\xi dx\nonumber\\
 \leq&\epsilon_2\int\rho \kappa(\theta)\theta_t^2dx+C\int
\kappa\theta_x^2dx+C\|u_x\|_{L^\infty}^2\int\rho\theta^2\kappa
dx\nonumber\\
&+\frac{d}{dt}\int(u_x^2+|\w_x|^2+|\bb_x|^2)\int_0^\theta \kappa(\xi)d\xi
dx\nonumber\\
&+C(\|u_x\|_{L^2}\|u_{xt}\|_{L^2}+\|\w_x\|_{L^2}\|\w_{xt}\|_{L^2}+\|\bb_x\|_{L^2}\|\bb_{xt}\|_{L^2})
\left\|\int_0^\theta \kappa(\xi)d\xi\right\|_{L^\infty}\nonumber\\
 \leq&\epsilon_2\int\rho \kappa(\theta)\theta_t^2dx+C\int
\kappa(\theta)\theta_x^2dx+C(1+\|u_{xx}\|_{L^2}^2)\int\rho\theta^2(1+\theta^q)dx\nonumber\\
&+\frac{d}{dt}\int(\lambda u_x^2+\mu |\w_x|^2+\nu |\bb_x|^2)\int_0^\theta
\kappa(\xi)d\xi
dx\nonumber\\
&+C(\|u_{xt}\|_{L^2}+\|\w_{xt}\|_{L^2}+\|\bb_{xt}\|_{L^2})\|\theta(1+\theta^q)\|_{L^\infty}\nonumber\\
 \leq&\epsilon_2\int\rho
\kappa(\theta)\theta_t^2dx+\epsilon_3(\|u_{xt}\|_{L^2}^2+\|\w_{xt}\|_{L^2}^2+\|\bb_{xt}\|_{L^2}^2)\nonumber\\
&+C\int
\kappa(\theta)\theta_x^2dx+C(1+\|u_{xx}\|_{L^2}^2)\nonumber\\
&+\frac{d}{dt}\int(\lambda u_x^2+\mu|\w_x|^2+\nu|\bb_x|^2) \int_0^\theta
\kappa(\xi)d\xi dx  +C\|\theta(1+\theta^q)\|_{L^\infty}^2, \label{bb26}
\end{align}
for any $0<\epsilon_2, \epsilon_3<1$.

Applying the operator $\partial_t$ to (\ref{aa3}) gives
\begin{equation}
\rho \w_{tt}+\rho u\w_{xt}-\mu\w_{xxt}=-\rho_t\w_t-\rho_tu\w_x-\rho
u_t\w_x+\bb_{xt}.\label{bb27}
\end{equation}
Multiplying (\ref{bb27}) by $w_t$, integrating the result over $\Omega$, and
using (\ref{aa1}) and Lemmas \ref{lebb5} and
  \ref{lebb6}, we have
\begin{align}
&\frac{1}{2}\frac{d}{dt}\int\rho |\w_t|^2dx+\int\mu |\w_{xt}|^2dx\nonumber\\
 =&-2\int\rho u \w_t\cdot \w_{xt}dx-\int\rho_t u \w_x\cdot \w_t dx\nonumber\\
 &-\int\rho u_t \w_x\cdot
\w_tdx-\int \bb_x\cdot \w_{xt}dx\nonumber\\
 \leq&2\|\sqrt\rho \w_t\|_{L^2}\|\sqrt\rho
u\|_{L^\infty}\|\w_{xt}\|_{L^2}+\|\rho_t\|_{L^2}\|u\|_{L^\infty}\|\w_x\|_{L^2}\|\w_t\|_{L^\infty}\nonumber\\
&+\|\sqrt\rho \w_t\|_{L^2}\|\sqrt\rho
u_t\|_{L^2}\|\w_x\|_{L^\infty}+\|\bb_x\|_{L^2}\|\w_{xt}\|_{L^2}\nonumber\\
 \leq&C\|\sqrt\rho
\w_t\|_{L^2}\|\w_{xt}\|_{L^2}+C\|\w_t\|_{L^\infty}\nonumber\\
&+C\|\sqrt\rho \w_t\|_{L^2}\|\sqrt\rho
u_t\|_{L^2}\|\w_{xx}\|_{L^2}+C\|\w_{xt}\|_{L^2}\nonumber\\
 \leq&\epsilon_4\|w_{xt}\|_{L^2}^2+C+C\|\sqrt\rho
\w_t\|_{L^2}^2+C\|\w_{xx}\|_{L^2}^2\|\sqrt\rho
u_t\|_{L^2}^2, \label{bb28}
\end{align}
for any $0<\epsilon_4<1$.

Applying the operator $\partial_t$ to (\ref{aa4}) gives
$$\bb_{tt}-\nu\bb_{xxt}=-(u \bb-\w)_{xt}.$$
Multiplying the above equation by $\bb_t$, integrating the result over $\Omega$,
and using Lemma \ref{lebb5}, we
find that
\begin{align}
&\frac{1}{2}\frac{d}{dt}\int |\bb_t|^2dx+\int \nu |\bb_{xt}|^2dx\nonumber\\
=& \int(u\bb-\w)_t\cdot \bb_{xt}dx\nonumber\\
 =&\int(u_t\bb+\bb_tu-\w_t)\cdot \bb_{xt}dx\nonumber\\
 \leq&(\|\bb\|_{L^\infty}\|u_t\|_{L^2}+\|u\|_{L^\infty}\|\bb_t\|_{L^2}+\|\w_t\|_{L^2})\|\bb_{xt}\|_{L^2}\nonumber\\
 \leq&C(\|u_t\|_{L^2}+\|\bb_t\|_{L^2}+\|\w_t\|_{L^2})\|\bb_{xt}\|_{L^2}\nonumber\\
 \leq&C(\|u_{xt}\|_{L^2}+\|\bb_t\|_{L^2}+\|\w_{xt}\|_{L^2})\|\bb_{xt}\|_{L^2}\nonumber\\
 \leq&\frac{\nu}{2}\|\bb_{xt}\|_{L^2}^2+C(\lambda \|u_{xt}\|_{L^2}^2+\mu\|\w_{xt}\|_{L^2}^2+ \|\bb_t\|_{L^2}^2).\label{bb29}
\end{align}

Combining (\ref{bb25}), (\ref{bb26}), (\ref{bb19}), and (\ref{bb28}) with
(\ref{bb29}), taking $\epsilon_i \,(i=1,\dots,4)$ small enough, integrating the
resulting inequality over $(0,t)$, then we conclude that
\begin{equation}
\int(\rho u_t^2+\rho
|\w_t|^2+|\bb_t|^2+\kappa^2(\theta)\theta_x^2)dx+\iint_{Q_T}
(\lambda u_{xt}^2+\mu|\w_{xt}|^2+\nu|\bb_{xt}|^2+\rho\theta_t^2)dx dt\leq
C.\label{bb30}
\end{equation}
where we have used the following estimate:
\begin{align*}
&\int_0^t\left(\frac{d}{dt}\int(\lambda u_x^2+\mu|\w_x|^2+\nu|\bb_x|^2)\int_0^\theta
\kappa(\xi)d\xi dx\right)d\tau\\
 \leq&\int(\lambda u_x^2+\mu|\w_x|^2+\nu|\bb_x|^2)dx\left\|\int_0^\theta
\kappa(\xi)d\xi\right\|_{L^\infty}+C\\
 \leq& C\left\|\int_0^\theta \kappa(\xi)d\xi\right\|_{L^\infty}+C\\
 \leq& C\left\|\left(\int_0^\theta
\kappa(\xi)d\xi\right)_x\right\|_{L^2}+C\\
 \leq& C\|\kappa(\theta)\theta_x\|_{L^2}+C\leq\epsilon_5\|\kappa(\theta)\theta_x\|_{L^2}^2+C,
\end{align*}
for any $0<\epsilon_5<1$.

It follows from (\ref{aa2}), (\ref{aa3}), (\ref{aa4}), (\ref{bb30}), and
Lemmas \ref{lebb5} and \ref{lebb6} that
\begin{equation*}
 \int(\lambda u_{xx}^2+\mu|\w_{xx}|^2+\nu|\bb_{xx}|^2)dx\leq C.
\end{equation*}
\end{proof}

Noting the above estimate,  (\ref{aa5}),
and (\ref{bb30}), it follows that
\begin{align*}
\|\theta_{xx}\|_{L^2}^2 \leq&C\int(\theta_x^4+u_x^4+|\w_x|^4+|\bb_x|^4
+\rho\theta_t^2+u^2\theta_x^2+\theta^2u_x^2)dx\\
 \leq& C+C\int\theta_x^4dx+C\int\rho\theta_t^2dx\\
 \leq& C+C\|\theta_x^2\|_{L^\infty}\int\theta_x^2dx+C\int\rho\theta_t^2dx\\
 \leq& C+C\|(\theta_x^2)_x\|_{L^1}+C\int\rho\theta_t^2dx\\
 \leq& C+C\|\theta_x\theta_{xx}\|_{L^1}+C\int\rho\theta_t^2dx\\
 \leq& C+C\|\theta_{xx}\|_{L^2}+C\int\rho\theta_t^2dx,
\end{align*}
which yields
\begin{equation}
\|\theta_{xx}\|_{L^2}^2\leq C+C\int\rho\theta_t^2dx.\label{bb31}
\end{equation}

\begin{Lemma}\label{lebb8}
\begin{equation}
\int(\rho_{xx}^2+\rho_{xt}^2)dx+\iint_{Q_T}(\rho_{tt}^2+u_{xxx}^2)dx
dt\leq C.\label{bb32}
\end{equation}
\end{Lemma}
\begin{proof} Applying the operator $\partial_x^2$ to (\ref{aa1}) gives
$$
\rho_{xxt}=-\rho_{xxx}u-3\rho_{xx}u_x-3\rho_xu_{xx}-\rho
u_{xxx}.
$$
Multiplying the above equation by $2\rho_{xx}$, integrating them    over
$\Omega$, and  using Lemmas
\ref{lebb7} and   \ref{lebb6}, we find that
\begin{align}
&\frac{d}{dt}\int\rho_{xx}^2dx\nonumber\\
 =&-5\int\rho_{xx}^2u_xdx-6\int\rho_x\rho_{xx}u_{xx}dx-2\int\rho\rho_{xx}u_{xxx}dx\nonumber\\
 \leq&5\|u_x\|_{L^\infty}\int\rho_{xx}^2dx+6\|\rho_x\|_{L^\infty}\|\rho_{xx}\|_{L^2}\|u_{xx}\|_{L^2}
+2\|\rho\|_{L^\infty}\|\rho_{xx}\|_{L^2}\|u_{xxx}\|_{L^2}\nonumber\\
 \leq&C\int\rho_{xx}^2dx+C\int u_{xxx}^2dx+C.\label{bb33}
\end{align}

Applying $\partial_x$ to (\ref{aa2}), integrating them    over
$\Omega$, and  using Lemmas \ref{lebb6}  and
  \ref{lebb7}, we infer that
\begin{align*}
& \|u_{xxx}\|_{L^2}\nonumber\\
 \leq&\left\|(\rho u)_{xt}+\left(\rho
u^2+P+\frac{1}{2}|\bb|^2\right)_{xx}\right\|_{L^2}\\
 \leq&\Big\|\rho_xu_t+\rho u_{xt}+\rho_xuu_x+\rho u_x^2+\rho
uu_{xx}+\rho_{xx}\theta+2\rho_x\theta_x\nonumber\\
& +\rho\theta_{xx}+\bb\cdot \bb_{xx}+|\bb_x|^2\Big\|_{L^2}\\
 \leq&\|\rho_x\|_{L^2}\|u_t\|_{L^\infty}+\|\rho\|_{L^\infty}\|u_{xt}\|_{L^2}+\|\rho_x\|_{L^2}
\|u\|_{L^\infty}\|u_x\|_{L^\infty}\\
&+\|\rho\|_{L^\infty}\|u_x\|_{L^4}^2 +\|\rho\|_{L^\infty}
\|u\|_{L^\infty}\|u_{xx}\|_{L^2}+\|\theta\|_{L^\infty}\|\rho_{xx}\|_{L^2}\\
&+2\|\rho_x\|_{L^\infty}\|\theta_x\|_{L^2}+\|\rho\|_{L^\infty}\|\theta_{xx}\|_{L^2}
+\|b\|_{L^\infty}\|b_{xx}\|_{L^2}+\|b_x\|_{L^4}^2\\
 \leq&C\|u_t\|_{L^\infty}+C\|u_{xt}\|_{L^2}+C+C\|\rho_{xx}\|_{L^2}+C\|\rho_x\|_{L^\infty}
+C\|\theta_{xx}\|_{L^2}\\
 \leq&C\|u_{xt}\|_{L^2}+C+C\|\rho_{xx}\|_{L^2}+C\|\theta_{xx}\|_{L^2}.
\end{align*}

Inserting the above estimates into (\ref{bb33}), and  using Lemma
\ref{lebb7}  and the Gronwall inequality, we get
$$\int\rho_{xx}^2dx+\iint_{Q_T}u_{xxx}^2dx dt\leq C.$$
Since $$\rho_{xt}=-(\rho u)_{xx},$$ it is easy to show that
\begin{align*}
\int\rho_{xt}^2dx \leq&C\int(\rho^2u_{xx}^2+\rho_x^2u_x^2+\rho_{xx}^2u^2)dx\\
 \leq&C\int(u_{xx}^2+\rho_{xx}^2)dx+C\|u_x\|_{L^\infty}^2\int\rho_x^2dx\leq
C.
\end{align*}

Finally, noting
\begin{align*}
\rho_{tt} = -(\rho u)_{xt}=-(\rho_tu+\rho u_t)_x
 = -(\rho_{xt}u+\rho_tu_x+\rho_xu_t+\rho u_{xt}),
\end{align*}
it holds that
\begin{align*}
\iint_{Q_T}\rho_{tt}^2dx dt
\leq&C\|u\|_{L^\infty(Q_T)}^2\iint_{Q_T}\rho_{xt}^2dx
dt+C\|\rho_t\|_{L^\infty(Q_T)}^2\iint_{Q_T}u_x^2dx dt\\
&+C\|\rho_x\|_{L^\infty(Q_T)}^2\iint_{Q_T}u_t^2dx
dt+C\|\rho\|_{L^\infty}^2\iint_{Q_T}u_{xt}^2dx dt\\
 \leq&C\iint_{Q_T}(\rho_{xt}^2+u_x^2+u_t^2+u_{xt}^2)dx dt\\
 \leq&C+C\iint_{Q_T}u_{xt}^2dx dt\leq C.
\end{align*}
\end{proof}

\begin{Lemma}\label{lebb9}
\begin{equation}
\int\rho\theta_t^2dx+\iint_{Q_T}|(\kappa(\theta)\theta_x)_t|^2dx
dt\leq C.\label{bb34}
\end{equation}
\end{Lemma}

\begin{proof}
Applying $\partial_t$ to (\ref{aa5}) gives
\begin{align*}
& \rho\theta_{tt}+\rho u\theta_{xt}-(\kappa(\theta)\theta_x)_{xt}
\\
=&
2(u_xu_{xt}+w_xw_{xt}+b_xb_{xt})-pu_{xt}-p_tu_x-\rho_t\theta_t
-\rho_tu\theta_x-\rho u_t\theta_x.
\end{align*}
Multiplying the above equation by
$\kappa(\theta)\theta_t=\left(\int_0^\theta
\kappa(\xi)d\xi\right)_t$, integrating them    over
$\Omega$, and using (\ref{aa1}),
$(\kappa\theta_t)_x=(\kappa\theta_x)_t$, Lemmas \ref{lebb7} and
\ref{lebb8}, we infer that
\begin{align}
&\frac{1}{2}\frac{d}{dt}\int\rho
\kappa(\theta)\theta_t^2dx+\int|(\kappa\theta_t)_x|^2dx\nonumber\\
 =&\frac{1}{2}\int\rho\theta_t^2(\kappa(\theta))_t dx+\frac{1}{2}\int\rho
u\theta_t^2(\kappa(\theta))_x
dx\nonumber\\&+2\int(u_xu_{xt}+\w_x\cdot \w_{xt}+\bb_x\cdot \bb_{xt})\kappa(\theta)\theta_t
dx\nonumber\\
&-\int(Pu_{xt}+P_tu_x+\rho_t\theta_t+\rho_tu\theta_x+\rho
u_t\theta_x)\kappa(\theta)\theta_t dx\nonumber\\
 \leq&\frac{1}{2}\|(\kappa(\theta))_t\|_{L^\infty}\int\rho\theta_t^2dx
+\frac{1}{2}\int\rho\theta_t^2dx\|u\|_{L^\infty}\|(\kappa(\theta))_x\|_{L^\infty}\nonumber\\
&+2(\|u_x\|_{L^2}\|u_{xt}\|_{L^2}+\|\w_x\|_{L^2}\|\w_{xt}\|_{L^2}
+\|\bb_x\|_{L^2}\|\bb_{xt}\|_{L^2})\|\kappa(\theta)\theta_t\|_{L^\infty}\nonumber\\
&+\left(\|P\|_{L^\infty}\|u_{xt}\|_{L^2}+\|\rho_t\|_{L^2}\|\theta\|_{L^\infty}
\|u_x\|_{L^\infty}\right.\nonumber\\
&+\left.\|\sqrt\rho\theta_t\|_{L^2}\|\sqrt\rho\|_{L^\infty}
\|u_x\|_{L^\infty}\right)\|\kappa(\theta)\theta_t\|_{L^\infty}\nonumber\\
&+\int(\rho u)_x \kappa
\theta_t^2dx+\|\rho_t\|_{L^2}\|u\|_{L^\infty}
\|\theta_x\|_{L^2}\|\kappa(\theta)\theta_t\|_{L^\infty}\nonumber\\
&+\|\sqrt\rho u_t\|_{L^2}\|\sqrt\rho\|_{L^\infty}
\|\theta_x\|_{L^2}\|\kappa(\theta)\theta_t\|_{L^\infty}\nonumber\\
\leq&C\|\kappa(\theta)\theta_t\|_{L^\infty}\int\rho\theta_t^2dx+C
\|\theta_{xx}\|_{L^2}\int\rho\theta_t^2dx\nonumber\\
&+C(\|u_{xt}\|_{L^2}+\|\w_{xt}\|_{L^2}+\|\bb_{xt}\|_{L^2})\|\kappa(\theta)\theta_t\|_{L^\infty}\nonumber\\
&+C(1+\|\sqrt\rho\theta_t\|_{L^2})\|\kappa(\theta)\theta_t\|_{L^\infty}-\int\rho
u(k\theta_t^2)_xdx.\label{bb35}
\end{align}
Noting that
\begin{align}
\|\kappa(\theta)\theta_t\|_{L^\infty} \leq& C\left(\int\rho
\kappa(\theta)|\theta_t|dx+\|(\kappa(\theta)\theta_t)_x\|_{L^2}\right)\nonumber\\
 \leq&C\left(1+\int\rho
\kappa(\theta)\theta_t^2dx+\|(\kappa(\theta)\theta_t)_x\|_{L^2}\right),\label{bb36}
\end{align}
and
\begin{align}
&-\int\rho u(\kappa(\theta)\theta_t^2)_x dx\nonumber\\
 =&-\int\rho u(\kappa(\theta) \theta_t)_x\theta_t dx-\int\rho u
\kappa(\theta) \theta_t\theta_{xt}dx\nonumber\\
 \leq&\|u\|_{L^\infty}\|\sqrt\rho\theta_t\|_{L^2}
\|\sqrt\rho(\kappa (\theta)\theta_t)_x\|_{L^2}\nonumber\\
&+\|u\|_{L^\infty}\|\sqrt\rho\theta_t\|_{L^2}\big\|\sqrt\rho[(\kappa(\theta)\theta_t)_x-
\kappa'(\theta)\theta_x\theta_t]\big\|_{L^2}\nonumber\\
 \leq&C\|\sqrt\rho\theta_t\|_{L^2}\|(\kappa(\theta) \theta_t)_x\|_{L^2}
+C\|\sqrt\rho\theta_t\|_{L^2}^2\|\theta_x\|_{L^\infty}\nonumber\\
 \leq&C\|\sqrt\rho\theta_t\|_{L^2}\|(\kappa(\theta)\theta_t)_x\|_{L^2}
+C\|\theta_{xx}\|_{L^2}\|\sqrt\rho\theta_t\|_{L^2}^2.\label{bb37}
\end{align}

Inserting (\ref{bb36}) and (\ref{bb37}) into (\ref{bb35}) and using
the Gronwall inequality, we arrive at (\ref{bb34}).
\end{proof}

\begin{Lemma}\label{lebb10}
\begin{equation}
\int\theta_{xx}^2dx+\iint_{Q_T}\theta_{xxx}^2dxdt\leq C.\label{bb38}
\end{equation}
\end{Lemma}
\begin{proof}
It follows from (\ref{bb31}) and (\ref{bb34}) that
\begin{equation}
\int\theta_{xx}^2dx\leq C.\label{bb39}
\end{equation}
It is easy to verify that
\begin{equation}
\int_0^T\|\theta_t\|_{L^\infty}^2dt\leq\int_0^T\|\kappa(\theta)\theta_t\|_{L^\infty}^2dt
\leq C\int_0^T\|(\kappa(\theta)\theta_t)_x\|_{L^2}^2dt+C\leq C.\label{bb40}
\end{equation}

Since
$$
\kappa(\theta)\theta_{xt}=(\kappa(\theta)\theta_t)_x-\kappa'(\theta)\theta_t\theta_x,
$$
by applying (\ref{bb40}) and Cauchy inequality, we have
\begin{align}
\iint_{Q_T}\theta_{xt}^2dx dt
\leq&C\iint_{Q_T}\kappa^2(\theta)\theta_{xt}^2dx
dt\nonumber\\
 \leq&C\iint_{Q_T}|(\kappa(\theta)\theta_t)_x|^2dx
dt+C\iint_{Q_T}(\kappa'(\theta))^2\theta_t^2\theta_x^2dx dt\nonumber\\
 \leq&C+C\int_0^T\|\theta_t\|_{L^\infty}^2\|\theta_x\|_{L^2}^2dt\nonumber\\
 \leq&C+C\int_0^T\|\theta_t\|_{L^\infty}^2dt\leq C,\label{bb41}
\end{align}

Applying the operator $\partial_x$ to (\ref{aa5}) gives
\begin{align*}
\kappa(\theta)\theta_{xxx} =&-3\kappa'(\theta)\theta_x\theta_{xx}-\kappa''(\theta)\theta_x^3-2(\lambda u_xu_{xx}
+\mu \w_x\cdot\w_{xx}+\nu\bb_x\cdot\bb_{xx})\\
&-\rho_x\theta_t-\rho\theta_{xt}-(\rho u\theta_x)_x-(\rho\theta
u_x)_x,
\end{align*}
whence
\begin{align}
\int\theta_{xxx}^2dx \leq&\int \kappa^2(\theta)\theta_{xxx}^2dx\leq
C\int\theta_x^2\theta_{xx}^2dx+C\int\theta_x^6dx\nonumber\\
&+C\int(u_x^2u_{xx}^2+|\w_x|^2|\w_{xx}|^2+|\bb_x|^2|\bb_{xx}|^2)dx\nonumber\\
&+C\|\rho_x\|_{L^\infty}^2\int\theta_t^2dx+C\|\rho\|_{L^\infty}^2
\int\theta_{xt}^2dx+C\|\rho
u\|_{L^\infty}^2\int\theta_{xx}^2dx\nonumber\\
&+C\|\rho_x\|_{L^\infty}^2\|u\|_{L^\infty}^2\|\theta_x\|_{L^2}^2
+C\|\rho\|_{L^\infty}^2\|u_x\|_{L^\infty}^2\|\theta_x\|_{L^2}^2\nonumber\\
&+C\|\rho_x\|_{L^\infty}^2\|\theta\|_{L^\infty}^2\|u_x\|_{L^2}^2
+C\|\rho\|_{L^\infty}^2\|\theta_x\|_{L^\infty}^2\|u_x\|_{L^2}^2\nonumber\\
&+C\|\rho\|_{L^\infty}^2\|\theta\|_{L^\infty}^2\|u_{xx}\|_{L^2}^2\nonumber\\
\leq&C+C\int\theta_t^2dx+C\int\theta_{xt}^2dx,\label{bb42}
\end{align}
by Lemma \ref{lebb7}, Lemma \ref{lebb8} and Lemma \ref{lebb9}.

The estimates (\ref{bb40}), (\ref{bb41}) and (\ref{bb42}) imply
$$\iint_{Q_T}\theta_{xxx}^2dxdt\leq C.$$
\end{proof}

By combining all the estimates obtained above, we get sufficient  a priori  estimates uniformly with $\delta$
 to take the limit $\delta \rightarrow 0^+$ and then extend the local strong solutions to be global one.
 Since the process is standard \cite{CH06,FYb}, we omit
 them here for brevity. Hence the proof of Theorem \ref{thaa1} is completed.

\bigskip
 \noindent
{\bf Acknowledgements:}
The authors are indebted to the referees for their careful reading  and valuable suggestions which improved the presentation of our paper.
 Fan was supported by NSFC (No. 11171154). Li was supported by NSFC (Grant Nos.  11271184, 11671193), China Scholarship Council and PAPD.

%\newpage

%%%%%%%%%%%-----------------%%%%%%%%%%%%%%%%%%%%%%%%%%%%%%%%%%%%%%%%%%%%%%%%%%

 \end{document}